\documentclass[11pt]{amsart}
\usepackage[curve,all]{xy}
 \xyoption{curve}
 \xyoption{line}
\xyoption{arc}
\usepackage{graphicx}
\usepackage{amssymb,amsmath,amscd,xy}
\usepackage[dvips]{epsfig}
\usepackage{srcltx}

\usepackage{xy}
\xyoption{matrix}

\usepackage{amsthm}

\addtocounter{section}{0}             % Start with section 1
\begin{document}
\setlength{\baselineskip}{0.47cm}        % Previous 0.57
%%%%%%%%%%%%%%%%%%%%%%%%%%%%%%%%%%%%%%%%%%%%%%%%%%%%%%%%%%%%%%%%%%%%

\newtheorem{theorem}{Theorem}
\newtheorem{problem}{Problem}
\newtheorem{lemma}{Lemma}
\newtheorem{proposition}{Proposition}
\newtheorem{remark}{Remark}
\newtheorem{property}{Property}
\newtheorem{definition}{Definition}
\newtheorem{corollary}{Corollary}
\newtheorem*{question}{Question}
\newtheorem{example}{Example}
\newtheorem{assumption}{Assumption}
\newtheorem*{notation}{Notation}
\newtheorem{claim}{Claim}

\title{Complements of graphs of meromorphic functions and complete vector fields}

\author{Alvaro Bustinduy}

\address{Departamento de Ingenier{\'\i}a Industrial \newline
         \indent Escuela Polit{\'e}cnica Superior \newline
         \indent Universidad Antonio de Nebrija \newline
         \indent C/ Pirineos 55, 28040 Madrid. Spain}
         \email{abustind@nebrija.es}

\thanks{Supported by Spanish MICINN project MTM2011-26674-C02-02}

\begin{abstract}
Given a meromorphic function $s: \mathbb{C}\to \mathbb{P}^{1}$, we
obtain a family of fiber-preserving dominating holomorphic maps
from $\mathbb{C}^{2}$ onto $\mathbb{C}^{2}\setminus
\textnormal{graph}(s)$ defined in terms of the flows of complete
vector fields of type $\mathbb{C}^{\ast}$ and  of an entire
function $h:\mathbb{C}\to\mathbb{C}$ whose graph does not meet
$\textnormal{graph}(s)$, which was determined by Buzzard and Lu. In
particular, we prove that the dominating map constructed by these
authors to prove the dominability of $\mathbb{C}^{2}\setminus
\textnormal{graph}(s)$ is in the above family. We also study the
complement of a double section in $\mathbb{C}\times\mathbb{P}^{1}$
in terms of a complex flow. Moreover,  when $s$ has at most one
pole, we prove that there are infinitely many complete vector
fields tangent to $\textnormal{graph}(s)$, describing explicit
families of them with all their trajectories proper and of the same
type ($\mathbb{C}$ or $\mathbb{C}^{\ast}$), if
$\textnormal{graph}(s)$ does not contain zeros; and families with
almost all trajectories non-proper and of type $\mathbb{C}$, or of
type $\mathbb{C}^{\ast}$, if $\textnormal{graph}(s)$ contains
zeros. We also study the dominability of $\mathbb{C}^{2}\setminus
A$ when $A\subset \mathbb{C}^{2}$ is invariant by the flow of a complete
holomorphic vector field.
\end{abstract}

%
%%%%%%%%%%%%%%%%%%%%%%%%%%%%%%%%%%%%%%%%%%%%%%%%%%%%%%%%%%%%%%%%%%
%

\maketitle \markright{}
\setcounter{tocdepth}{1}
\tableofcontents

%\newpage

%
%%%%%%%%%%%%%%%%%%%%%%%%%%%%%%%%%%%%%%%%%%%%%%%%%%%%%%%%%%%%%%%%%%
%

\section{Introduction}

A complex manifold $M$ of dimension two (a surface) is (holomorphically)
dominable (by $\mathbb{C}^{2}$) if there is a holomorphic map
$f:\mathbb{C}^{2} \to M$ with Jacobian determinant not identically
zero. The map $f$ is called a dominating map. Note that, in general, $f$ might
be non-surjective. For instance, $M=\mathbb{C}^{2}$ and
$f(x,y)=(x,xy)$. As interesting property, one easily obtains that if
$M$ is dominable, there exists a holomorphic map from $\mathbb{C}$
to $M$ whose image is not contained in any complex subvariety of
dimension one of $M$, what is called a holomorphic image
of $\mathbb{C}$ in $M$ which is not analytically degenerated. In particular,
this property implies that $M$ can not be hyperbolic.

The study of surfaces dominable by $\mathbb{C}^{2}$ has been
developed in the last years by G. Buzzard and S. Lu, see \cite{BL}
and \cite{BL2} (references therein for classical results).
In many cases, the existence of a holomorphic
image of $\mathbb{C}$ in $M$ which is not analytically degenerated
implies the dominability of $M$ (see \cite[\S 3 and \S
4]{BL} for $M$ compact, and \cite[\S 5]{BL} for  $M$ non-compact
and algebraic).  In general, it is not easy to know if the
complement of a given analytic curve $C$ in a non-hyperbolic $M$
is dominable or not. Even if $M=\mathbb{C}^{2}$,
$\mathbb{C}\times\mathbb{P}^1$ or $\mathbb{P}^{2}$ the question is
not entirely solved. On the other hand, one sees that the study of dominability of
the complement in $\mathbb{P}^{2}$ of a smooth cubic \cite[\S
5.1]{BL}  (see \S 2), and in particular, the
complement in $\mathbb{C}^{2}$ of the graph of a meromorphic
function, are behind of the methods to construct explicit
dominating maps in the algebraic setting  \cite{BL}  and in other
contexts, as the complement in $\mathbb{C}\times \mathbb{P}^{1}$
of a double section \cite{BL2}.

%%%%%%%%%%%%%%%%%%%%%%%%%

In this paper, we use complete vector fields to study dominability problems in the two dimensional case.
We will recall some definitions (for more details, see \cite{Bustinduy-indiana}, \cite{Suzuki-anales}, and references therein).
Let $X$ be a holomorphic vector field on $M$.
Associated to $X$ there is a differential equation:
\begin{equation*}
\varphi'_{z}(t)=X(\varphi_z(t)),\,\,\,\,\,
\varphi_z(0)=z\in M,
\end{equation*}
whose local solutions $\varphi_z$ define the local flow of $X$ in
a neighbourhood of $(0,z)\in\mathbb{C}\times M$.
Given any point $z\in M$, the local solution
$\varphi_{z}$ can be extended by analytic continuation along paths
in $\mathbb{C}$, beginning from $t=0$, to a maximal connected Riemann surface $\pi_{z}: \Omega_z \to \mathbb{C}$, which is a Riemann
domain over $\mathbb{C}$. The map ${\varphi}_z:\Omega_z \to M$ is said to be a solution and its image $C_z$
is called the trajectory of $X$ through $z$. We say that a trajectory $C_z$ of $X$ is proper if its topological closure $\overline{C}_z$ in $M$
defines an analytic curve in $M$ of pure dimension
one.

If
$\Omega_z$ is equal to $\mathbb{C}$ (as domain in $\mathbb{C}$)  for all $z$ in $M$, we say that $X$ is complete. In this case,
each  trajectory $C_z$ is a Riemann surface uniformized by $\mathbb{C}$.
If $M$ is Stein (in particular if $M=\mathbb{C}^2$):

\begin{itemize}

\item [{$-$}] Any trajectory of $X$ is of type (= analytically isomorphic to) $\mathbb{C}$ or $\mathbb{C}^{\ast}$.

  \item [{$-$}]There is a pluripolar set $E\subset M$ invariant by the flow of $X$ such that every complex trajectory $C_{z}$ with $z\in M\setminus E$ is of the same type. In particular, we say that $X$ is of type $\mathbb{C}$ (resp. $\mathbb{C}^{\ast}$) if $C_{z}$ is of type $\mathbb{C}$ (resp. $\mathbb{C}^{\ast}$) for a $z\in M\setminus E$.

  \item  [{$-$}] Any trajectory of type $\mathbb{C}^{\ast}$ is proper. Moreover, if $X$ is of type $\mathbb{C}^{\ast}$, there is a meromorphic first integral for $X$, and hence all the trajectories of $X$ are proper.
\end{itemize}

%%%%%%%%%%%%%%%%%%%%%%%%%%%%%%%%

Let us summarize the results of this article by sections.

\noindent $\bullet$ Section 2. \\
\noindent In \cite[Theorem\,5.2]{BL}, it is proved that the complement
of the graph of a meromorphic function $s:\mathbb{C}\to
\mathbb{P}^{1}$ in $\mathbb{C}^{2}$ is dominable. In fact, an
explicit fiber-preserving surjective dominating map $\Phi$ from
$\mathbb{C}^{2}$ to $\mathbb{C}^{2}\setminus
\textnormal{graph}(s)$ is constructed in terms of an entire
function $h:\mathbb{C}\to\mathbb{C}$ whose graph does not meet
$\textnormal{graph}(s)$ and the entire map $\Psi(t,w)=
(e^{tw}-1)/t$ (see (II) of $\S2$). One of the main motivations of
this article is to see if $\Phi$ can be given using a complex
flow.

We define a family $\{Z^{u}\}_{u\in\mathbb{C}(z)}$ of complete
vector fields on $\mathbb{C}^{2}\setminus \textnormal{graph}(s)$,
of type $\mathbb{C}^{\ast}$, whose trajectories are contained in
vertical lines (Proposition\,\ref{P1}). It allows us to obtain a
family $\{f^{u}\}_{u\in\mathbb{C}(z)}$ of fiber-preserving
surjective dominating maps from $\mathbb{C}^{2}$ to
$\mathbb{C}^{2}\setminus \textnormal{graph}(s)$, where each
$f^{u}$ is defined in terms of the above $h\in\mathbb{C}(z)$ and
of the complex flow of $Z^{u}$. Moreover, by integration of
$Z^{u}$, we can  explicitly obtain $f^{u}$ and show that the map
$\Phi$ constructed by Buzzard and Lu in \cite{BL} is one of the
dominating maps of $\{f^{u}\}_{u\in\mathbb{C}(z)}$
(Theorem~\ref{T1}).

\noindent $\bullet$ Section 3. \\
Let $D$ be a double section over $\mathbb{C}$ in $\mathbb{C}\times \mathbb{P}^{1}$.
We define a family $\{W^{u}\}_{u\in\mathbb{C}(z)}$ of complete
vector fields on $(\mathbb{C}\times \mathbb{P}^{1})\setminus D$,
of type $\mathbb{C}^{\ast}$, whose trajectories are contained in
vertical lines (Proposition\,\ref{P2}). We obtain a family
$\{g^{u}\}_{u\in\mathbb{C}(z)}$ of fiber-preserving surjective
dominating maps  from $\mathbb{C}^{2}$ to $(\mathbb{C}\times
\mathbb{P}^{1})\setminus D$, where each $g^{u}$ is defined in
terms of the complex flow of $W^{u}$ and of a holomorphic function
$\sigma:\mathbb{C}\to \mathbb{P}^{1}$, which is determined in \cite[Theorem\,1.3]{BL2}, whose graph does not meet
$D$ (Theorem~\ref{T2}).

As a corollary, we obtain
an alternative proof of \cite[Theorem\,1.2]{BL2}, using only $\sigma$
and the completeness of $W^{u}$ (Corollary~\ref{C2}).

\noindent $\bullet$ Section 4. \\
Let $s:\mathbb{C}\to\mathbb{P}^{1}$ be a meromorphic function.
From Section 2, it follows that each vector field $Z^{u}$ on
$\mathbb{C}^{2}\setminus \textnormal{graph}(s)$ is complete and
never vanishes. One natural question is to ask whether there is a
complete holomorphic vector field $X$ on $\mathbb{C}^{2}$
different from $Z^{u}$ such that when it is restricted to
$\mathbb{C}^{2}\setminus \textnormal{graph}(s)$ is complete and
never vanishes. There are two possibilities for such an $X$ to be
analyzed: $X$ is identically zero on $\textnormal{graph}(s)$, or
it is not. In this article we will study only the latter one. In
particular, $X$ has at most isolated zeros, and $X$ is tangent to
$\textnormal{graph}(s)$.

First, we will prove that $s$ has at most one pole
(Lemma\,\ref{Superficies R}).

We study complete vector fields on $\mathbb{C}^{2}$ tangent to
$C=\textnormal{graph}(s)$. We prove that there are infinitely many
complete vector fields $X$, tangent to $C$, and without zeros on
$C$. In this case, all the trajectories of $X$ are proper and of
type $\mathbb{C}$, when $s$ has no poles ($C$ of type
$\mathbb{C}$), or of type $\mathbb{C}^{\ast}$, when $s$ has one
pole ($C$ of type $\mathbb{C}^{\ast}$). In the former case, the
vector field is analytically equivalent, by a fiber-preserving
automorphism of $\mathbb{C}^{2}$ that takes
$\textnormal{graph}(s)$ into $\{t=0\}$, to a polynomial vector
field. In the latter case, the vector field is analytically
equivalent, by a fiber-preserving automorphism of $\mathbb{C}^{2}$
that takes $\textnormal{graph}(s)$ into
$\textnormal{graph}(1/z^{k})$, to a polynomial vector field. For
$s$ without poles, and fixed $p\in C$, we prove that there exist
infinitely many complete holomorphic vector fields $X$, tangent to
$C$ and with only one zero on $C$ at $p$. In this case, moreover,
$X$ can be defined of type $\mathbb{C}$ and with almost all its
trajectories non-proper, or of type $\mathbb{C}^{\ast}$ ($\S$4.1,
$\S$4.2, Theorem\,\ref{T3}).

\noindent $\bullet$ Section 5. \\
Let $A$ be a subset of  $\mathbb{C}^{2}$ invariant by the flow of an holomorphic vector field $X$.
We study, in some cases, the dominability of $\mathbb{C}^{2}\setminus A$ when $X_{\mid A}$ is
complete. If $A$ is an analytic curve transversal to the foliation
defined by $X$ we determine that $A$ is the graph
of a meromorphic function, after an analytic automorphism, and hence $\mathbb{C}^{2}\setminus A$ is
dominable (Proposition\,\ref{P4} and Theorem\,\ref{T4}). If $X$ is
a polynomial vector field, with isolated singularities, and $A=C_{z}$ is
a proper trajectory  of type $\mathbb{C}^{\ast}$ where
$\overline{C}_{z}$ is a singular curve, we prove that there is a
 dominating holomorphic map $\Gamma$ from $\mathbb{C}^{2}$ to
$\mathbb{C}^{2}\setminus C_{z}$, such that
$\Gamma(\mathbb{C}^{2})$ is biholomorphic to
$\left({\mathbb{C}}^{2}\setminus
\{xy(y^{r}-ax^{s})=0\}\right)\cup\{(0,0)\}$, with $r$, $s\in\mathbb{N}^{+}$, $rs \neq 1$ and $(r,s)=1$ (Theorem\,\ref{T5}).

\subsection*{Acknowledgements}
I want to thank Professor Miguel S\'anchez Caja and the Universidad de Granada for
their invitation, during which I could finish this article. Thanks as well
to Professors Junjiro Noguchi and Takeo Ohsawa to let me know the work of Buzzard and Lu.
Finally, we also want to thank the referee for his suggestions that have improved
this paper a lot.

%\subsection{Statement of the theorem}

\section{Complement of the graph of a meromorphic function}

\subsection{Dominability of  the complement of a smooth cubic  in $\mathbb{P}^{2}$ }
Let $A$ be a smooth cubic in $\mathbb{P}^{2}$ and let
$X=\mathbb{P}^{2}\setminus A$. In \cite[Proposition 5.1]{BL} G.
Buzzard and S. Lu proved that $X$ is holomorphically dominable by $\mathbb{C}^{2}$. The proof is based
in two points:

\begin{enumerate}
\item [{\bf (I)}] The existence of a meromorphic function $s:\mathbb{C}\to\mathbb{P}^{1}$, which is associated to $A$,
such that dominability of $\mathbb{C}^{2}\setminus \textnormal{graph}(s)$  implies dominability of $X$.
\item [{\bf (II)}] The construction of an explicit fiber-preserving
holomorphic map $\Phi$ from $\mathbb{C}^{2}$ onto
$\mathbb{C}^{2}\setminus \textnormal{graph}  (s)$ with
Jacobian determinant not identically zero.
\end{enumerate}

In fact, the proof of (II) implies that the complement in $\mathbb{C}^{2}$ of
any graph of a meromorphic function $s:\mathbb{C}\to
\mathbb{P}^{1}$ must be holomorphically dominable by
$\mathbb{C}^{2}$ \cite[Theorem 5.2]{BL}.

Let us summarize (II) (see \cite[p.\,645]{BL} for precise
details). Note that we can assume that $s \neq s_{\infty}$. In what follows, we will denote by $\mathbb{C}(z)$  the ring of entire functions of one variable.
To prove {\bf (II)}, there are two points:

\medskip

\noindent {\bf (II.1)} \em Existence of $h\in\mathbb{C}(z)$ such that  $\textnormal{graph}  (h)\cap \textnormal{graph}  (s)=\emptyset$. \em \\
\noindent
If $s=q/q_{1}$ for $q,\,q_{1}\in\mathbb{C}(z)$
 without common zeros, it is enough to define  $h=s-1/g$ where
$g=q_{1}/e^{g_{1}}$ for $g_{1}\in\mathbb{C}(z)$ such that $1/g$
and $s$ have the same principal parts. The existence of $g_{1}$  follows from Mittag-Leffler and Weierstrass Theorems as we can see in \cite[p.\,645]{BL}.

\medskip

\noindent {\bf (II.2)} \em Explicit definition of $\Phi$ using (II.1) and  $\Psi(t,w)= (e^{tw}-1)/t$.\em \\
Note that
$\Psi(t,w)$ is entire on $\mathbb{C}^{2}$ because
$$\Psi(t,w)=w+\frac{tw^{2}}{2!}+\frac{t^{2}w^{3}}{3!}+\dots$$
Let us take
$$\phi(z,w)=h(z)-\Psi(g(z),w)=h(z)- \frac{e^{g(z)w}-1}{g(z)}=s(z)-\frac{e^{g(z)w}}{g(z)}.$$
Note that $\phi(z,w)$ is entire on $\mathbb{C}^{2}$ and equal to $h(z) - w$  if $g(z)=0$. It holds that
$$\Phi(z,w)=(z,\phi(z,w))=\left (z,s(z)-\frac{e^{g(z)w}}{g(z)}\right)$$
is a fiber-preserving dominating map from
$\mathbb{C}^{2}$ onto $\mathbb{C}^{2}\setminus \textnormal{graph}
(s)$ with non-vanishing Jacobian determinant.

\medskip

\subsection{Buzzard-Lu's results on $\mathbb{C}^{2}\setminus \textnormal{graph}(s)$ revisited}

\subsubsection{Family
$\{Z^{u}\}_{u\in\mathbb{C}(z)}$ of complete vector fields on
$\mathbb{C}^{2}\setminus \textnormal{graph}(s)$}

Let us see in the following proposition that there is a natural family of complete vector fields on $\mathbb{C}^{2}\setminus \textnormal{graph}(s)$.
\begin{proposition}\label{P1}

Let $s:\mathbb{C}\to\mathbb{P}^{1}$ be a meromorphic function.
There exists a family $\{Z^{u}\}_{u\in\mathbb{C}(z)}$ of vector fields
on $\mathbb{C}^{2}\setminus\textnormal{graph}(s)$, which are complete, of type $\mathbb{C}^{\ast}$, and
whose trajectories are contained in vertical lines. Moreover, each $Z^{u}$ extends to $\mathbb{C}\times\mathbb{P}
^{1}$, as complete vector field of type $\mathbb{C}^{\ast}$,
vanishing only along the graphs of $s$ and ${s}_{\infty}$.\\
\end{proposition}
\begin{proof}
Let us take $s=q/q_{1}$ for $q,\,q_{1}\in\mathbb{C}(z)$ without common zeros. For any
$u\in\mathbb{C}(z)$, we define
$$Z^{u}= e^{u(z)}\cdot [q_{1}(z)w - q_{1}(z) s(z) ]\frac{\partial}{\partial w}.$$
It is easy to check the following facts, which imply the proof of proposition:

$\bullet$ The trajectories of $Z^{u}$ are contained in vertical lines.

$\bullet$ $Z^{u}$ is holomorphic on $\mathbb{C}^{2}$ and only vanishes along $\textnormal{graph}(s)$.

$\bullet$ $Z^{u}$ restricted to $\{z=z_{0}\}$, with $z_{0}\in\mathbb{C}$, is a complete linear vector field. Then $Z^{u}$ is complete on $\mathbb{C}^{2}\setminus\textnormal{graph}(s)$.

$\bullet$
The trajectory of $Z^{u}$ contained in $\{z=z_{0}\}$ is  either of type $\mathbb{C}$,  if $q_{1}(z_{0})=0$; or of type $\mathbb{C}^{\ast}$, if $q_{1}(z_{0})\neq 0$.
Therefore, $Z^{u}$ is of type $\mathbb{C}^{\ast}.$

$\bullet$ $Z^{u}$ extends to $\mathbb{C}\times \mathbb{P}^{1}$ as holomorphic vector field (also denoted by $Z^{u}$) vanishing on $\{w=\infty\}$.
Thus $Z^{u}$ is complete on $\mathbb{C}\times \mathbb{P}^{1}$, and of type $\mathbb{C}^{\ast}$.
\end{proof}
Let $Z^{u}$ be a vector field  of $\{Z^{u}\}_{u\in\mathbb{C}(z)}$. The meromorphic function $${\mathcal{P}}^{u}(z,w)={\left(\frac{2\pi
i}{e^{u(z)}q_{1}(z)}\right)}^{2}$$ on $\mathbb{C}^{2}$ is the
period function of $Z^{u}$. Thus for any $(z,w)$ with $q_{1}(z)\neq 0$,
the trajectory of $Z^{u}$ through $(z,w)$, of type
$\mathbb{C}^{\ast}$, has period $\sqrt{{\mathcal{P}}^{u}}$. That
is, $\sqrt{{\mathcal{P}}^{u}}\mathbb{Z}$ is the discrete subgroup
of $(\mathbb{C},+)$ defined by the complex times that fix $(z,w)$ by
the flow of $Z^{u}$ (details on period function, see \cite[page 84]{Suzuki-springercorto}).

\medskip

Let us see in the following remark that (II.1) can be interpreted in terms of a vector field of
$\{Z^{u}\}_{u\in\mathbb{C}(z)}$.
\begin{remark}\em
Let $s=q/q_{1}$ and $h=s-1/g$ for $g,q,\,q_{1}\in\mathbb{C}(z)$ as
in (II.1), where $g=q_{1}/e^{g_{1}}$ with $g_{1}\in\mathbb{C}(z)$.

The vector field $Z^{-g_{1}}$ of $\{Z^{u}\}_{u\in\mathbb{C}(z)}$,
which is defined by
$$e^{-g_{1}}\cdot [q_{1}(z)w - q_{1}(z) s(z) ]\frac{\partial}{\partial w}= [g(z)w - g(z) s(z) ]\frac{\partial}{\partial w},$$
has a period function ${\mathcal{P}}^{-g_{1}}=(2\pi i / g)^{2}$.
It holds that $s-\sqrt{{\mathcal{P}}^{-g_{1}}}/2\pi i=h$. Note that
to determine $h$ as in (II.1) it is enough to determine
$u\in\mathbb{C}(z)$ such that $s-\sqrt{{\mathcal{P}}^{u}}/2\pi
i\in\mathbb{C}(z)$.

\em

\end{remark}

\vspace{0.015cm}

\subsubsection{Family
$\{f^{u}\}_{u\in\mathbb{C}(z)}$ of fiber-preserving surjective
dominating maps} We will denote by ${\varphi}^{u}$ the global flow
of $Z^{u}$ . Recall that ${\varphi}^{u}$ is a map from $\mathbb{C}
\times(\mathbb{C}\times \mathbb{P}^{1})$ to $\mathbb{C}\times
\mathbb{P}^{1}$ defined by $\varphi^{u}(t,z,w)=(z(t),w(t))$, where
$(z(t),w(t))$ is the solution of $Z^{u}$ through $z(0)=z$ and
$w(0)=w$.

Let us see in the following theorem that there is a family of
fiber-preserving dominating holomorphic maps from $\mathbb{C}^{2}$
onto $\mathbb{C}^{2}\setminus \textnormal{graph}(s)$. These
holomorphic maps will be defined using  $h\in\mathbb{C}(z)$ of
(II.1) and the complex flows of the vector fields of
$\{Z^{u}\}_{u\in\mathbb{C}(z)}$.

\begin{theorem}
\label{T1} Let $s:\mathbb{C}\to\mathbb{P}^{1}$ be a meromorphic
function. Let us consider $h\in\mathbb{C}(z)$ as in (II.1) and the complex flow ${\varphi}^{u}$ of $Z^{u}$.
Then, the family
$$\{f^{u }\}_{u\in\mathbb{C}(z)}$$
of holomorphic maps
$$f^{u}(z,t)=\varphi^{u}(t,z,h(z)),$$
is a family of  fiber-preserving dominating maps from $\mathbb{C}^{2}$ onto
$\mathbb{C}^{2}\setminus\textnormal{graph}(s)$.\\
\noindent  To be more precise, let $s=q/q_{1}$ and  $h=s-1/g$ for
$g,q,\,q_{1}\in\mathbb{C}(z)$ as in (II.1). Then, $f^{u}$ is
explicitly given by
$$f^{u}(z,t)= \left( z,
s(z)-\dfrac{e^{[e^{u(z)}q_{1}(z)]t}}{g(z)}\right).
$$

\end{theorem}
\begin{proof} By definition, each $f^{u}$ is holomorphic. As $h(z)\neq \infty $ and $h(z)\neq s(z)$ then
$(z,h(z))$ is not in the set of zeros of $Z^{u}$ by
Proposition~\ref{P1}. Thus, by completeness of $Z^{u}$, any point
$(z,w)\in\mathbb{C}^{2}\setminus \textnormal{graph}(s)$ can be
reached by the solution $(z(t),w(t))$ of $Z^{u}$, with $z(0)=z$
and $w(0)=h(z)$, after time $t$, and $f^{u}$ is surjective. The
vanishing of $\partial/\partial z$-component of $Z^{u}$ implies
that $f^{u}$ is fiber-preserving, and $Z^{u}(z,h(z))\neq{0}$
implies that the Jacobian determinant of $f^{u}$ is not
identically zero, since it is $e^{u(z)}[q_{1}(z)h(z)-q(z)]\neq{0}$
at $(z,0)$.

Let us obtain the explicit expression of $f^{u}$. We take $(z,h(z))$ in $\mathbb{C}^{2}\setminus
\textnormal{graph}(s)$. Then $z(t)\equiv z$. The second component
of $X$ gives the linear differential equation
$$dw/dt=
[e^{u(z)}q_{1}(z)]w(t) - [e^{u(z)}q_{1}(z)]s(z),\,\,w(0)=w=h(z).
$$
If $g(z)\neq{0}$, the above equation can be explicitly solved, and
we get
\begin{alignat*}{3}
w(t) = & \left\{ h(z)-\int_{0}^{t} [e^{u(z)}q_{1}(z)]s(z)\,
e^{-\left[\int_{0}^{x}[e^{u(z)}q_{1}(z)]ds \right]}dx \right\}
e^{\int_{0}^{t}[e^{u(z)}q_{1}(z)]ds}=\\
& \\
             = &  \left\{ h(z)-    [e^{u(z)}q_{1}(z)]s(z) \int_{0}^{t} e^{-[e^{u(z)}q_{1}(z)]x} dx \right\} e^{[e^{u(z)}q_{1}(z)]t}=\\
& \\
             = & [h(z)-s(z)]e^{[e^{u(z)}q_{1}(z)]t}+s(z)=s(z)-e^{[e^{u(z)}q_{1}(z)]t}/g(z).
\end{alignat*}
Note that $w(t)$ (by definition, it is entire) is equal to
$w-h(z)$ when $g(z)=0$.
\end{proof}

\noindent
As a consequence of Theorem~\ref{T1}, let us prove in the following corollary that the dominating map $\Phi$
constructed by Buzzard and Lu can be defined using $h\in\mathbb{C}(z)$ of (II.1) and the complex flow ${\varphi}^{-g_{1}}$ of  $Z^{-g_{1}}$.

\begin{corollary} Given a meromorphic
function $s:\mathbb{C}\to\mathbb{P}^{1}$ suppose  that $s=q/q_{1}$
and $h=s-1/g$ for $g,q,\,q_{1}\in\mathbb{C}(z)$ as in (II.1),
where $g=q_{1}/e^{g_{1}}$ with $g_{1}\in\mathbb{C}(z)$. Then, the
domimating map $f^{-g_{1}}=\varphi^{-g_{1}}(t,z,h(z))$ of
$\{f^{u}\}_{u\in\mathbb{C}(z)}$ is the dominating map $\Phi(z,t)$
constructed by Buzzard an Lu of (II.2).
\end{corollary}
\begin{proof}
It follows from the explicit expression of $f^{u}$ obtained in
Theorem~\ref{T1}.
\end{proof}

%\begin{remark}
%\em
%Note that for $s=q/q_{1}$ and $g=q_{1}/e^{g_{1}}$ as (II.1)
%$s(z)-e^{[e^{u(z)}q_{1}(z)]t}/g(z)$ is entire in $\mathbb{C}^{2}$, independently of $u(z)$.
%\em
%\end{remark}

\section{Complements of double sections}

A  complex subvariety $D\subset R \times {\mathbb{P}}^{1}$ of dimension one
is a double section over a noncompact Riemann surface $R$
if $D=\{(z,w)| a(z)w^2 + b(z)w+ c(z) = 0\},$
where $a$, $b$, $c$ are holomorphic functions. Let us assume
that $R=\mathbb{C}$.

\subsection{Buzzard-Lu's results on $(\mathbb{C}\times\mathbb{P}^{1})\setminus D$ revisited}
G. Buzzard and S. Lu  proved in \cite[\S 2]{BL2} that
$(\mathbb{C}\times \mathbb{P}^{1})\setminus D$ is holomorphically
dominable. To prove it, they consider three steps:

\begin{enumerate}
\item [{\bf (1)}] The existence of a holomorphic function
$\sigma:\mathbb{C}\to \mathbb{P}^{1}$ whose graph does not meet
$D$  \cite[Theorem\,1.3]{BL2}.

\item [{\bf (2)}] The existence of a change of coordinates in
$\mathbb{C}\times \mathbb{P}^{1}$ that transforms  $\sigma$ into
$s_{\infty}$ so that $D$ is $\{(z,w)| w^2  = h(z)\}$
\cite[Corollary 1.4]{BL2}.

\item [{\bf (3)}] The explicit construction of a fiber-preserving
dominating  map from $\mathbb{C}^{2}$ to $(\mathbb{C}\times
\mathbb{P}^{1})\setminus D$ \cite[Theorem\,1.2]{BL2} using (1) and (2).
\end{enumerate}

\subsubsection{Family
$\{W^{u}\}_{u\in\mathbb{C}(z)}$ of complete vector fields on
$(\mathbb{C}\times \mathbb{P}^{1})\setminus D$}

Let us see in the following proposition that there is a natural family of complete vector fields on $(\mathbb{C}\times \mathbb{P}^{1})\setminus D$.

\begin{proposition}\label{P2}
Let $D$ be a double section over $\mathbb{C}$.
There exists a family $\{W^{u}\}_{u\in\mathbb{C}(z)}$ of vector fields  on
$(\mathbb{C}\times \mathbb{P}^{1})\setminus D$, which are complete, of type $\mathbb{C}^{\ast}$, and
whose trajectories are contained in vertical lines. Moreover, each $W^{u}$ vanishes only along $D$.
\end{proposition}
\begin{proof}

For any $u\in\mathbb{C}(z)$, one defines
$$W^{u}=e^{u(z)} \cdot [a(z)w^2 + b(z)w+ c(z) ]\frac
{\partial}{\partial w}. $$ It is easy to check the following facts, which imply the proof of the proposition:

$\bullet$  The trajectories of $W^{u}$ are contained in vertical lines.

$\bullet$ $W^{u}$ is holomorphic on $\mathbb{C}\times
\mathbb{P}^{1}$ and it only vanishes along $D$.

$\bullet$ $W^{u}$ on $\{z=z_{0}\} \cup \{ \infty\}\simeq
\mathbb{P}^{1}$ is a complete holomorphic vector field, with (one
or two) zeros where $D$ intersects $\{z=z_{0}\}\cup \{ \infty\}$.
Therefore, $W^{u}$ is complete on $(\mathbb{C}\times
\mathbb{P}^{1})\setminus D$.

$\bullet$ The trajectory of $W^{u}$ contained in $\{z=z_{0}\}$,
with $z_{0}\in\mathbb{C}$, is of type $\mathbb{C}^{\ast}$, if $D$
intersects $\{z=z_{0}\}$ at two points or of type $\mathbb{C}$ if
$D$ intersects  $\{z=z_{0}\}$ at one point. In particular, $Z^{u}$ is of
type $\mathbb{C}^{\ast}$.\end{proof}

\subsubsection{Family
$\{g^{u}\}_{u\in\mathbb{C}(z)}$ of fiber-preserving surjective
dominating maps}
One can define analogously as in Theorem~\ref{T1}
a family of fiber-preserving dominating holomorphic maps from
$\mathbb{C}^{2}$ onto $(\mathbb{C}\times \mathbb{P}^{1})\setminus D$ using (1) and the flow of $W^{u}$.
The proof of the following theorem is left to the reader.

\begin{theorem} \label{T2} Let $D\subset\mathbb{C}\times \mathbb{P}^{1}$
be a double section over $\mathbb{C}$. Let us consider the
holomorphic map $\sigma:\mathbb{C}\to \mathbb{P}^{1}$ whose graph does not meet $D$, according to (1), and the flow ${\psi}^{u}$ of $W^{u}$.
Then, the
family
$$\{g^{u}\}_{u\in\mathbb{C}(z)}$$ of holomorphic maps
$$g^{u}(z,t)=\psi^{u}(t,z,\sigma(z)),$$ is a family of
fiber-preserving dominating maps from $\mathbb{C}^{2}$ onto
$(\mathbb{C}\times \mathbb{P}^{1})\setminus D$.
\end{theorem}

As a corollary we obtain an alternative proof of \cite[Theorem\,1.2]{BL2}.

\begin{corollary}\label{C2}
Let $D\subset\mathbb{C}\times \mathbb{P}^{1}$ be a double section
over $\mathbb{C}$. Then, the explicit fiber-preserving
dominating map from $\mathbb{C}^{2}$ to
$(\mathbb{C}\times\mathbb{P}^{1})\setminus D$ defined in (3) can be
similarly obtained using (only) (1) and the flow of any $W^{u}$ of
$\{W^{u}\}_{u\in\mathbb{C}(z)}$.
\end{corollary}

\subsubsection{Some applications}

\begin{proposition}
Let $s:\mathbb{C}\to\mathbb{P}^{1}$ be a meromorphic function. Then,

\begin{enumerate}
  \item[(i)] There exist $h\in\mathbb{C}(z)$ and a meromorphic function
$\hat{s}:\mathbb{C} \to \mathbb{P}^{1}$ such that
$\textnormal{graph}(s)\cap \textnormal{graph}(h)$,
$\textnormal{graph}(h)\cap \textnormal{graph}(\hat{s})$ and
$\textnormal{graph}(\hat{s})\cap \textnormal{graph}(s)$ are empty.
  \item[(ii)]
   There exists a complete vector field
$\hat{W}$ on $\mathbb{C}\times\mathbb{P}^{1}$, vanishing only
along $\textnormal{graph}(h)\cup\textnormal{graph}(\hat{s})$,
whose trajectories are all of type $\mathbb{C}^{\ast}$.
  \item[(iii)] There exist  a double section $\hat{D}\subset\mathbb{C}\times \mathbb{P}^{1}$
over $\mathbb{C}$ and a fiber-preserving holomorphic map from $\mathbb{C}^{2}$ onto
$(\mathbb{C}\times\mathbb{P}^{1})\setminus \hat{D}$ of the form
$F(z,t)=(z,H(z,t))$, such that $F$ on each vertical fiber
$\{z=z_{0}\}$ is of the form  $e^{\,c_{z_{0}}t}$, with a constant
$c_{z_{0}}\neq{0}$ depending on $z_{0}$, modulo an automorphism
of $\mathbb{P}^{1}$ mapping the two points
in $\hat{D}\cap \{z=z_{0}\}$ to  $0$ and $\infty$.
\end{enumerate}
\end{proposition}
\begin{proof}
Let $s=q/q_{1}$ and $h=s-1/g$ for $g,q,\,q_{1}\in\mathbb{C}(z)$ as
in (II.1). To deduce (i), let $\hat{s}=\sigma$, where
$\sigma:\mathbb{C}\to\mathbb{P}^{1}$ is the map whose graph does
not meet the double section $D=\{q_{1}(z)(w-h(z))(w-s(z))=0\}$
\cite[Theorem\,1.3]{BL2}.

To prove (ii),
take $\hat{s}=\hat{q}/\hat{q}_{1}$ with $\hat{q}$, $\hat{q}_{1}\in\mathbb{C}(z)$ without common zeros,
and define $\hat{W}=[\hat{q}_{1}(z)(w-h(z))(w-\hat{s}(z))]\partial/\partial w.$

To prove (iii), define the double section $\hat{D}\subset\mathbb{C}\times \mathbb{P}^{1}$
over $\mathbb{C}$ given by the set of
zeros of $\hat{W}$, and $F(z,t)=\hat{\varphi}(t,z,s(z))$ with
$\hat{\varphi}$ the flow of $\hat{W}$. \em
\end{proof}

\section{Complete vector fields tangent to graphs}

Let $X$ be a holomorphic vector field on $\mathbb{C}^{2}$.
We will denote by $\mathcal{F}(X)$ the holomorphic foliation defined by $X$. Note that $\mathcal{F}(X)$ is a foliation by curves on $\mathbb{C}^{2}$ with isolated singularities.

Let $s:\mathbb{C}\to\mathbb{P}^{1}$ be a meromorphic function. From Proposition\,\ref{P1}, it follows that each
vector field $Z^{u}$ on $\mathbb{C}^{2}\setminus
\textnormal{graph}(s)$ is complete and never vanishes. Note that  $\mathcal{F}(Z^{u})$ is defined
by vertical lines and is transversal to $\textnormal{graph}(s)$.

One natural question is to ask whether there is a complete
holomorphic vector field $X$ on $\mathbb{C}^{2}$ different from
$Z^{u}$ such that when it is restricted to
$\mathbb{C}^{2}\setminus \textnormal{graph}(s)$ is complete and
never vanishes. There are two possibilities for such an $X$ to be
analyzed.

$\bullet$ If $X$ is identically zero on $\textnormal{graph}(s)$ we can assume that $\textnormal{graph}(s)$ is
invariant by  $\mathcal{F}(X)$, for otherwise (the proof of) Proposition\,\ref{P4} implies that $X$ is of the form $Z^{u}$ after an analytic automorphism.
If moreover $X$ verifies one of the following properties:
\begin{enumerate}
  \item [{-}] $X$ defines a quasi-algebraic flow \cite{Suzuki-anales},
  \item [{-}] $X$ has a meromorphic first integral \cite{Suzuki-anales}, or
  \item [{-}] $X$ is polynomial \cite{Brunella-topology},
\end{enumerate}
one concludes that $s$ has at most one pole because its set of zeros is an algebraic curve in
$\mathbb{C}^{2}$ with components of type $\mathbb{C}$ or
$\mathbb{C}^{\ast}$, and thus $\textnormal{graph}(s)$ is of type
$\mathbb{C}$ or $\mathbb{C}^{\ast}$.

On the other hand, all known complete vector fields verify one of
the above properties. Nevertheless, in general, it is unknown
whether $s$ can have more than one pole.

\begin{question}
Is there a complete holomorphic vector field $X$ on $\mathbb{C}^{2}$,
identically zero on $\textnormal{graph}(s)$, such that $\textnormal{graph}(s)$ is invariant by $\mathcal{F}(X)$,
when $s$ has more than one pole?
\end{question}

$\bullet$ If $X$ is not identically zero on $\textnormal{graph}(s)$, $X$ has
at most isolated zeros on $\textnormal{graph}(s)$. Then,  $X$ is tangent to
$\textnormal{graph}(s)$, and $\textnormal{graph}(s)$ is invariant
by $\mathcal{F}(X)$. \em In this section, we will work under these
assumptions. \em

Let us see in following lemma that $s$ has at most one
pole and that $\textnormal{graph}(s)$ is of type $\mathbb{C}$ or $\mathbb{C}^{\ast}$.

\begin{lemma}\label{Superficies R} Let $s:\mathbb{C}\to\mathbb{P}^{1}$ be a meromorphic function. If there is
a complete holomorphic vector field $X$ on $\mathbb{C}^{2}$, which is tangent to $\textnormal{graph}(s)$ and with at most isolated zeros
on $\textnormal{graph}(s)$, then $s$ has at most one pole. Let us consider the trajectory $C_{z}$ of $X$ contained in $\textnormal{graph}(s)$, then

\noindent $-$ If $s\in\mathbb{C}(z)$ (no poles), $C_{z}$ is of type $\mathbb{C}$
or $\mathbb{C}^{\ast}$,
being $C_{z}$ respectively equal to $\textnormal{graph}(s)$ or $\textnormal{graph}(s)\setminus\{p\}$,
where $p$ is the unique zero of $X$ over $\textnormal{graph}(s)$, or\\
\noindent
$-$ If $s$
has only one pole,
$C_{z}$ is equal to $\textnormal{graph}(s)$ and of type $\mathbb{C}^{\ast}$.\em
\end{lemma}
\begin{proof}
Let $P\subset \mathbb{C}^{2}$ be the set of zeros of $X$.
Then $C_{z}=\textnormal{graph}(s)\setminus P$ is of type $\mathbb{C}$ or $\mathbb{C}^{\ast}$.
As $X_{|C_{z}}$ is complete, it extends as a complete holomorphic vector field
on $\textnormal{graph}(s)$ making zeros on $\textnormal{graph}(s)\cap P$.

If $C_{z}\simeq \mathbb{C}$, then $\textnormal{graph}(s)\cap P$ is empty. Otherwise
$\textnormal{graph}(s)\simeq \mathbb{P}^{1}$, which is not possible. Then $C_{z}=\textnormal{graph}(s)$ and  $s\in\mathbb{C}(z)$.

If $C_{z}\simeq \mathbb{C}^{\ast}$, then $\textnormal{graph}(s)\cap P$ is empty or has only one point $p$. Otherwise $X$ is not complete on $\textnormal{graph}(s)$. In the first case,
$C_{z}=\textnormal{graph}(s)$ and $s$ has only one pole. In the second case, $C_{z}\cup\{p\}=\textnormal{graph}(s)\simeq\mathbb{C}$
and  $s\in\mathbb{C}(z)$.
\end{proof}

\medskip

\subsection{Vector fields tangent to $\textnormal{graph}(s)$ if  $s\in\mathbb{C}(z)$}
Let us construct two families of complete holomorphic vector fields on $\mathbb{C}^{2}$ tangent to $\textnormal{graph}(s)$ when $s\in\mathbb{C}(z)$.

\begin{proposition}\label{TtipoC}
Let $s:\mathbb{C}\to \mathbb{C}$ be a holomorphic map. Then,
there are two families of complete vector fields, which are tangent to $\textnormal{graph}(s)$ and
analytically equivalent by a fiber-preserving automorphism that takes $\textnormal{graph}(s)$ into
the line $\{t=0\}$
to one of the polynomial vector fields: \\
\noindent $\mathbf{(i)}$
$$
X_{1}= (ax +b)\frac{\partial}{\partial x} + A(x)t
\frac{\partial}{\partial t}
$$
where $a,b\in \mathbb{C},$ $A\in\mathbb{C}[x]$,
$A\not\equiv{0}$, or \\
\noindent $\mathbf{(ii)}$
$$
X_{1}=at\frac{\partial}{\partial t} + A(x^{m}t^{n})\cdot \left(    nx
\frac{\partial}{\partial x}  - mt \frac{\partial}{\partial t}
\right)
$$
where $a\in\mathbb{C}$, $m,n\in\mathbb{N}^{+}$, $(m,n)=1$,
$A\in\mathbb{C}[y]$, $A\not\equiv{0}$, and $y=x^{m}t^{n}$.
\end{proposition}
\begin{proof}
A polynomial vector field $X_{1}$ as $(i)$ or $(ii)$ is complete
\cite[Theorem]{Brunella-topology} and leaves invariant $\{t=0\}$.
On the other hand, if $\phi$ is the automorphism  given by
$\phi(x,t)=(x,t+s(x))=(z,w)$, since $\phi(\{t=0\})=
\textnormal{graph}(s)$, it is enough to define the two families as
$\phi_{\ast} X_{1}$, with $X_{1}$ as $(i)$ or $(ii)$.
\end{proof}
\begin{proposition}\label{TtipoC1}
Let $s:\mathbb{C}\to \mathbb{C}$ be a holomorphic map. Then, there
is a family of complete vector fields  tangent to
$\textnormal{graph}(s)$, which have all their trajectories proper
and of type $\mathbb{C}$.
\end{proposition}
\begin{proof}
It is sufficient to take  $\phi_{\ast} X_{1}$ as in proof of
Proposition\,\ref{TtipoC}, with $X_{1}$ as $(i)$, $a=0$ and
$b\neq{0}$.
\end{proof}

\begin{proposition}\label{TtipoC2}
Let us consider $s\in\mathbb{C}(z)$ and $p\in
\textnormal{graph}(s)$. Then, there is a family of complete
holomorphic vector fields tangent to $\textnormal{graph}(s)$ and
vanishing only at $p$. Moreover, infinitely many of them are of type $\mathbb{C}$ and with almost all their
trajectories non-proper, and infinitely many other of them are of type $\mathbb{C}^{\ast}$. Therefore,
there are infinitely many vector fields of this family not equivalent by an analytic
automorphism.
\end{proposition}

\begin{proof}
Let us take $X=\phi_{\ast} X_{1}$ as in proof of
Proposition\,\ref{TtipoC}, with $X_{1}$ as $(i)$, where $a\neq{0}$, $b=0$ and $A(0)\neq{0}$;
or with $X_{1}$ as $(ii)$, where
$a\neq{0}$, $A(0)\neq{0}$ and $a-A(0)m\neq {0}$.  Let $\lambda$ be the
quotient of the eigenvalues of the linear part of $DX(z,w)$ at the
unique zero of $X$, that we can assume as $p$. Note that in the former case
$\lambda=a/A(0)$, and in the latter case $\lambda=A(0)n / (a-A(0)m)$. Let us take
$a$ and $A(0)$ so that $\lambda\not \in \mathbb{R}^{-}\cup\{0\}$.
According to Poincare's Theorem, the foliation $\mathcal{F}(X)$ is
given by $x\partial /
\partial x + \lambda y \partial / \partial y$ in suitable
coordinates $(x,y)$ around $p$
\cite[p.\,10]{Brunella-impa}. In particular, in a neighborhood $U_{p}\subset \mathbb{C}^{2}$ of $p$,
if $C_z$ is a trajectory with $z\in U_{p}$,  $C_{z}\cap U_{p}$  contains a level set of $x^{-\lambda}y$.

\noindent $\bullet$
If $\lambda\in\mathbb{C}\setminus\mathbb{R}$, for any $z\in U_{p}$,
$C_{z}\cap U_{p}$ accumulates $\{xy=0\}$ \cite[p.\,120]{Loray}. Hence, $C_{z}$ is not proper and of type $\mathbb{C}$,
and $X$ is of type $\mathbb{C}$ (see $\S 1$).

\noindent $\bullet$
If $\lambda\in\mathbb{R}\setminus\mathbb{Q}$, for any $z\in U_{p}$, $C_{z}\cap U_{p}$
contains a real subvariety of dimension three \cite[p.\,120]{Loray}.
Hence, $C_{z}$ is not proper and of type $\mathbb{C}$, and $X$ is of type $\mathbb{C}$ (see $\S 1$).

\noindent $\bullet$ If  $\lambda=p/q\in\mathbb{Q}^{+}$, for any $z\in U_{p}$, $C_{z}\cap U_{p}$
is a punctured disk. Hence, $C_{z}$ is proper and of type $\mathbb{C}^{\ast}$, and $X$ is of type $\mathbb{C}^{\ast}$ (see $\S 1$).

Finally, if $\beta(z,w)=(u,v)$ is an analytic automorphism of $\mathbb{C}^{2}$, since the quotient of the eigenvalues
of the linear part $DY(u,v)$ of $Y={\beta}_{\ast}X$ at $p'=\beta(p)$ is $\lambda$, the last sentence of the statement follows. \end{proof}

\subsection{Vector fields tangent to $\textnormal{graph}(s)$ if  $s$ has one pole}
Let us construct two families of complete holomorphic vector
fields on $\mathbb{C}^{2}$ tangent to $\textnormal{graph}(s)$ when
$s$ has one pole.
\begin{theorem}\label{TtipoCast}
Let $s:\mathbb{C}\to\mathbb{P}^{1}$ be a meromorphic function with
only one pole of order $k\in\mathbb{N}^{+}$. Then, there are two
families of complete vector fields of type $\mathbb{C}^{\ast}$,
which are tangent to $\textnormal{graph}(s)$, and analytically
equivalent by a fiber-preserving automorphism that takes
$\textnormal{graph}(s)$ into $\textnormal{graph}(1/z^{k})$
to one of the polynomial vector fields: \\
\noindent $\mathbf{(iii)}$
$$
Y_{0}=a z \frac{\partial}{\partial z} + \\
\left( A(z)w + \frac{- (ak + A(z))}{z^{k}}\right)
\frac{\partial}{\partial w},
$$
with $$A(z)= -ak + A_{0}(z),\,\,\,A_{0}\in z^{k}\cdot
\mathbb{C}[z],\,\,\,or$$

\noindent $\mathbf{(iv)}$
\begin{multline*}
Y_{0}=
a \left( \frac{wz^{k}-1}{z^{k}}\right) \frac{\partial}{\partial w}+ A(z^{m-nk}{(wz^{k}-1)}^{n})\cdot\\
 \cdot\left[
nz  \frac{\partial}{\partial z} + \left( - mw + \frac{(m -
nk)}{z^{k}} \right) \frac{\partial}{\partial w} \right].
\end{multline*}
with $m > nk$ and $$A(y)= a/ (m-nk) + A_{0}(y),\,\,\,A_{0}\in
y^{k}\cdot \mathbb{C}[y],\,\,y=z^{m-nk}{(wz^{k}-1)}^{n}.$$
\end{theorem}

\medskip

\begin{remark}
\em If $Y_{0}$ is as $(iii)$ or $(iv)$ of
Theorem\,\ref{TtipoCast}, according to \cite[Th\'eor\`eme
4]{Suzuki-anales} (see also $\S$5.1), there is an analytic
automorphism $\phi$ of $\mathbb{C}^{2}$ such that
$${\phi}_{\ast} Y_{0} = \lambda(t) \cdot \left(z  \frac{\partial}{\partial z}  -k w
\frac{\partial}{\partial w}\right),\,\,\lambda
\in\mathbb{C}(t),\,t=wz^{k}.$$

\noindent As a consequence of Theorem\,\ref{TtipoCast}, there are
two complete holomorphic vector fields of type $\mathbb{C}^{\ast}$
tangent to $\textnormal{graph}(s)$, generically transversal, and such that
their corresponding foliations are analytically  but not algebraically equivalent.

\em
\end{remark}

\subsection*{Proof of Theorem\,\ref{TtipoCast}}

We will assume that $z=0$ is the pole of $s$. Then
$s(z)=s_{0}(z)/z^{k}$, with $s_{0}\in\mathbb{C}(z)$,
$s_{0}(0)\neq{0}$ and $k\in\mathbb{N}^{+}$.
Let us see in the
following lemma that (II.1) allows us to define an automorphism that
takes $\textnormal{graph}(s)$ into $\textnormal{graph}(1/z^{k})$.

\begin{lemma}\label{LBL}
Under the conditions stated above, there exists a fiber-preserving automorphism $\varphi$ of $\mathbb{C}^{2}$ such that $$\varphi(\textnormal{graph}(s))=\textnormal{graph}(1/z^{k}).$$
\end{lemma}
\begin{proof}
Let us define
$\varphi (z,w)=(z,e^{-g_{1}(z)}[w-h(z)])=(x,y),$
where $h=s-1/g$ for $g=z^{k}e^{-g_{1}}\in\mathbb{C}(z)$ by (II.1).
Note that $\varphi$ is an automorphism and that
$\varphi^{-1}(x,y)=(x,e^{g_{1}(x)}y+ h(x))=(z,w).$
By definition, $\varphi(z,s(z))=(z,1/z^{k})$.
\end{proof}

Let us see in the following proposition an application of
Lemma\,\ref{LBL}.

\begin{proposition} Let $Z$ be a holomorphic vector field on $\mathbb{C}^{2}$. Let us suppose that $Z$ is tangent to
the curve $\{wz^{k}-s_{0}(z)=0\}$, with $s_{0}\in\mathbb{C}(z)$,
$s_{0}(0)\neq{0}$ and $k\in\mathbb{N}^{+}$. Then, $Z$ is
analytically equivalent, by a fiber-preserving analytic
automorphism $\varphi$ of $\mathbb{C}^{2}$, to a vector field $Y$
tangent to $\{wz^{k}-1=0\}$.
\end{proposition}
\begin{proof}
It is enough to define $Y=\varphi_{\ast} Z$, with $\varphi$ as in
Lemma\,\ref{LBL}.
\end{proof}

Let us see in the following proposition that the vector fields of Theorem\,\ref{TtipoCast} can be defined using
$\varphi$ of  Lemma~\ref{LBL}.

\begin{proposition}\label{FP} Let
$\varphi$ be the automorphism of Lemma~\ref{LBL}. Then, the vector
field $X=\varphi^{\ast} Y_{0}$, where $Y_{0}$ is as $(iii)$ or
$(iv)$ of Theorem\,\ref{TtipoCast}, is complete on
$\mathbb{C}^{2}$, of type $\mathbb{C}^{\ast}$, and tangent to
$\textnormal{graph}(s)$.
\end{proposition}
\begin{proof}
Let $X_{1}$ be as (i), with $b=0$, or as (ii) of Proposition~\ref{TtipoC}. Then $X_{1}$
leaves invariant $\{t=0\}$ and $\{x=0\}$. If  $\alpha$ is the map
given by $\alpha(x,t)=(x,t+1/x^{k})=(z,w),$ one has that $\alpha$
is a biholomorphism of $\mathbb{C}^{\ast}\times \mathbb{C}$ such
that  $\alpha(\{t=0\})= \textnormal{graph}(1/z^{k})$. Then
$\alpha_{\ast} X_{1}$ is a complete holomorphic vector field on
$\mathbb{C}^{\ast}\times\mathbb{C}$ tangent to
$\textnormal{graph}(1/z^{k})$. On the other hand,
\begin{equation*}
\alpha_{\ast} X_{1}= \left(
\begin{array}{cc}
1 & 0 \\
      \\
\dfrac {-k}{z^{k+1}}  & 1
\end{array}
\right)\cdot X_{1}(z,w-1/z^{k}).
\end{equation*}
An explicit computation shows that $\alpha_{\ast} X_{1}$ is
holomorphic on $\mathbb{C}^{2}$, and then complete on
$\mathbb{C}^{2}$, if $\alpha_{\ast} X_{1}=Y_{0}$, with $Y_{0}$ as in the statement. On the other
hand, since the quotient of the eigenvalues of the linear part of
$DY_{0}(p)$ at the unique zero $p$ of $Y_{0}$ is $-1/k$, $Y_{0}$
is of type $\mathbb{C}^{\ast}$, and then $X$ is as well. Finally,
from Lemma~\ref{LBL}, it follows that $X$ must be complete and
tangent to $\textnormal{graph}(s)$ since
$\varphi^{-1}(\textnormal{graph}(1/z^{k}))=\textnormal{graph}(s)$.
\end{proof}
Then, we have finished the proof of Theorem\,\ref{TtipoCast}.

\subsection{Summary}

The main results of $\S4.1$ and $\S4.2$ may be summarized in the following theorem:

\begin{theorem}\label{T3}
Let $C$ be an analytic curve defined by the graph of a meromorphic fuction $s:\mathbb{C}\to \mathbb{P}^{1}$ with at most one pole.
Then,\\
\noindent $\bullet$ If $s$ has no poles,\\
\noindent $-$ There exist infinitely many complete holomorphic
vector fields tangent to $C$, without zeros on $C$, and such that all their
trajectories are proper and of type $\mathbb{C}$.
Moreover, if $X$ is one of them, then $C$ is a trajectory $C_{z}$ of $X$ and $X$ is analytically equivalent, by a fiber-preserving analytic
automorphism of $\mathbb{C}^{2}$ that takes $C_{z}$ into a line, to a polynomial vector field.
\\
\noindent $-$
For any $p\in C$, there exist infinitely many complete holomorphic vector fields tangent to $C$, with only one zero on $C$, which is $p$, and not
equivalent by an analytic automorphism. If $X$ is one of them,
$C\setminus\{p\}$ is a trajectory $C_{z}$ of $X$ of type
$\mathbb{C}^{\ast}$. Moreover, $X$ can be defined of type
$\mathbb{C}$ and with almost all its trajectories non-proper, or
of type $\mathbb{C}^{\ast}$. In particular,
there are infinitely many of these vector fields not equivalent by an analytic
automorphism.\\
\noindent $\bullet$ If $s$ has one pole,\\
There exist infinitely many complete holomorphic vector fields
tangent to $C$, without zeros on $C$, and such that all their
trajectories are proper and of type $\mathbb{C}^{\ast}$.
Moreover, if $X$ is one of them, then $C$ is a trajectory $C_{z}$ of $X$ and $X$ is
analytically equivalent,
by a fiber-preserving analytic
automorphism of $\mathbb{C}^{2}$ that takes $C$ into  $\{wz^{k}-1=0\}$, to a polynomial vector field.
\em
\end{theorem}
\begin{remark}\em Under the assumption of this section, that is,
$X$ is not identically zero on $\textnormal{graph}(s)$,
the fact that $\textnormal{graph}(s)$ must be of type $\mathbb{C}$ or $\mathbb{C}^{\ast}$ by Lemma\,\ref{Superficies R} is the obstruction
to define complete vector fields tangent to $\textnormal{graph}(s)$
when $s$ has more than one pole.

The construction of an automorphism using (II.1), as $\varphi$ in Lemma\,\ref{LBL},
that takes $\textnormal{graph}(s)$ into $\textnormal{graph}(1/q_{1})$ also works when $s$ has more than one pole.
However, in this case our procedure only produces vector fields of the form $Z^{u}$.
The reason is that  $\alpha(x,t)=(x,s(x)+t)$ is a biholomorphism of $\mathbb{C}^{2}$ minus
at least two vertical lines (see $\S4.2$), and then the only complete polynomial vector field that we can
push forward by $\alpha$ is the vertical one (Picard Theorem).

On the other hand, the existence of a complete holomorphic vector field $X$ identically zero on $\textnormal{graph}(s)$, and such that
$\textnormal{graph}(s)$ is invariant by $\mathcal{F}(X)$,
when $s$ has more than one pole, would imply that there are other complete vector fields until now unknown, as we have mentioned at the beginning of this section.

\em
\end{remark}

\section{Dominability and complete vector fields}
Let $A$ be a subset of  $\mathbb{C}^{2}$ invariant by the flow of
a holomorphic vector field $X$. In this section we will study, in
some cases, the dominability of $\mathbb{C}^{2}\setminus A$ when
$X_{\mid A}$ is complete.

\subsection{$\mathbb{C}^{2}\setminus C$ for $C$ transversal to $\mathcal{F}(X)$ and $X$ complete on $\mathbb{C}^{2}\setminus C$.}

\begin{proposition}
\label{P4} Let $C$ be an analytic curve in $\mathbb{C}^{2}$. If
$C$ is not invariant by $\mathcal{F}(X)$ and $X$ is complete on
$\mathbb{C}^{2}\setminus C$ then:
\begin{enumerate}
\item[{$(i)$}] $X$ vanishes  on $C$, and then  $X$ is complete and
of type $\mathbb{C}^{\ast}$, \item[{$(ii)$}] Up to an analytic
automorphism of $\mathbb{C}^{2}$, $C$ is the graph of a
meromorphic function $s:\mathbb{C}\to\mathbb{P}^{1}$.
\end{enumerate}

\end{proposition}

\begin{proof}
To obtain $(i)$, assume that $X_{\mid C}\not\equiv {0}$ and
derive a contradiction. Take $p\in C$ such that $X(p)\neq{0}$. Suppose that the trajectory $C_{z}$ of $X$ through $p$ is transversal to $C$ at $p$. Locally, in a neighborhood $V$ of $p$, taking
coordinates $(z,w)$ with $p=(0,0)$, $X_{\mid V}=
\partial/\partial z$. On the other side, $X_{\mid
C_{z}\setminus C}$ complete implies that $X_{\mid (C_{z}\setminus C)\cap
V}$ has a zero at $0$, which is not possible. We conclude that $X_{\mid C}\equiv 0$
and $X_{\mid V}= h \partial/\partial z$, for $h$
holomorphic on $V$, vanishing on $C \cap V $, which we assume $\{ z=0 \}$, at
order $1$ by completeness.
Therefore $X$ has infinitely many trajectories of type $\mathbb{C}^{\ast}$ whose topological boundary
in $\mathbb{C}^{2}$ is one point in $C$. In particular,
$X$ is complete and of type $\mathbb{C}^{\ast}$, and defining a proper flow
 \cite{Suzuki-anales}.

To prove $(ii)$, the structure of $X$ is well known \cite[Th\'eor\`eme 4]{Suzuki-anales}.
Up to an analytic change of coordinates, $X$ is one of the following vector fields:
\begin{enumerate}

\item
$$a(z)\frac{\partial}{\partial w},$$
with $a\in\mathbb{C}(z)$.
\item
$$[g(z)(w-s(z))] \frac{\partial}{\partial w},$$
with $g\in\mathbb{C}(z)$, and $s$  meromorphic such that $gs\in\mathbb{C}(z)$.

\item $$\lambda (t) \cdot \left(n z  \frac{\partial}{\partial z}  + m w
\frac{\partial}{\partial w}\right),$$
with $\lambda(t)\equiv \lambda \in\mathbb{C}^{\ast}$, $m, n \in \mathbb{N}^{\ast}$, $(m,n)=1$ or $\lambda\in\mathbb{C}(t)$, $t=z^{-m}w^{n}$, $-m, n \in \mathbb{N}^{\ast}$, $(m,n)=1$.
\item
$$
\frac{\gamma (t) }{z^{\ell}}\cdot\left( nz^{\ell+1} \frac{\partial}{\partial z} -
[(m+n\ell)z^{\ell} w+mp(z)+nz \dot{p} (z)] \frac{\partial}{\partial w}\right),
$$
with $m,n \in\mathbb{N}^{\ast}$, $(m,n)=1$, $\ell\in \mathbb{N}$, $p\in\mathbb{C}[z]$ of degree
$<\ell$ with $p(0)\neq{0}$ if $\ell>0$ or $p\equiv 0$ if $\ell=0$, $\gamma\in\mathbb{C}(t)$ vanishing at $t=0$ at order  $\geq \ell/m$, $t=z^m(z^{\ell}w+p(w))^n.$
\end{enumerate}
Cases $(1)$, $(3)$ and $(4)$ are not possible because their set of
zeros are invariant by $\mathcal{F}(X)$. Therefore, $X$ is as
$(2)$ and $C$ is equal to $\textnormal{graph}(s)$.
\end{proof}

\begin{remark}\label{R1}\em
Note that $(ii)$ implies that $C$ is biholomorphic to $\mathbb{C}$
minus the set of poles of a meromorphic function in $\mathbb{C}$.  \em
\end{remark}

\begin{theorem}
\label{T4} Let $C$ be an analytic curve in $\mathbb{C}^{2}$. If
$C$ is not invariant by $\mathcal{F}$ and $X$ is complete on
$\mathbb{C}^{2}\setminus C$ then $\mathbb{C}^{2}\setminus C$ is
holomorphically dominable.
\end{theorem}
\begin{proof}
It follows from Theorem~\ref{T1} and Proposition~\ref{P4}.
\end{proof}

\subsection{$\mathbb{C}^{2}\setminus C_z$ for $\overline{C}_{z}$ singular and $X$ polynomial and complete on $C_{z}$}
Let us  study the dominability of the complementary in $\mathbb{C}^{2}$ of a trajectory $C_{z}$ of a vector field $X$. Here,
we only treat the case of a polynomial vector field $X$ with isolated zeros that is complete on a proper trajectory $C_z$. Recall that $C_{z}$ is proper if
the topological closure $\overline{C}_{z}$ in $\mathbb{C}^{2}$ is an analytic curve.
In this situation, $C_{z}$ is of type $\mathbb{C}$ or $\mathbb{C}^{\ast}$.
The trajectory $C_{z}$ is algebraic if $\overline{C}_{z}$ is an algebraic curve.

\subsubsection{$C_{z}$ of type $\mathbb{C}$}

Note that $\overline{C}_{z}=C_{z}$.
If $C_{z}$ is algebraic, ${C}_{z}=\{y=0\}$, after a polynomial automorphism by Abhyankar-Moh-Suzuki theorem \cite{Suzuki-japonesa}.
If $C_{z}$ is nonalgebraic, $C_{z}$ defines a leaf of an algebraic foliation $\mathcal{F}(X)$
with all its ends (one) planar, isolated and properly embedded in $\mathbb{C}^{2}$. Then
${C}_{z}=\{y=0\}$ after an analytic automorphism \cite{Brunella-topology2}. Then, $\mathbb{C}^{2}\setminus C_{z}$ is biholomophic to $\mathbb{C}\times \mathbb{C}^{\ast}$
(complement of graph of $s=0$), and $\mathbb{C}^{2}\setminus C_{z}$ is dominable.

\subsubsection{$C_{z}$ of type $\mathbb{C}^{\ast}$}
Note that $\overline{C}_{z}=C_{z}$ or $\overline{C}_{z}=C_{z}\cup \{p\}$ with $X(p) = 0$.
Let us study the case where $\overline{C}_{z}$ is a singular curve. Necessarily, $\overline{C}_{z}$ has only one
singularity, say $p$, and $\overline{C}_{z}=C_{z}\cup \{p\}$ with $X(p)=0$.

\begin{theorem}
\label{T5}
Let $C_{z}$ be a a proper trajectory of type $\mathbb{C}^{\ast}$ of a polynomial vector field $X$ on $\mathbb{C}^{2}$.
If $\overline{C}_{z}$ is singular and $X_{\mid C_{z}}$ is complete, then $\overline{C}_{z}=\{y^{r}-ax^{s}=0\},$ with
$a\neq{0}$, $r,s\in\mathbb{N}^{+}$, $r\cdot s\neq{1}$ and $(r,s)=1$, after an analytic automorphism. In particular,
\begin{enumerate}
\item[{(i)}]
There is a non-surjective holomorphic dominating map $\Gamma$ from $\mathbb{C}^{2}$ to $\mathbb{C}^{2}\setminus C_{z}$, such that
\item[{(ii)}]
Its image $\Gamma(\mathbb{C}^{2})$ is biholomorphic to
$$\left({\mathbb{C}}^{2}\setminus \{xy(y^{r}-ax^{s})=0\}\right)\cup\{(0,0)\}.$$
\end{enumerate}
\end{theorem}

\medskip

\subsection*{Proof of Theorem\,\ref{T5}}

If $C_{z}$ is algebraic, $\overline{C}_{z}$ is defined by $\{y^{r}-ax^{s}=0\}$ after a polynomial automorphism, by Lin-Zaidenberg's
theorem \cite{Lin-zaidenberg}.

If $C_{z}$ is nonalgebraic, let us see in the following proposition that
$\overline{C}_{z}$ is also given by above equation after an analytic automorphism.

\begin{proposition}\label{P6}
Let $C_{z}$ be a proper trajectory of type $\mathbb{C}^{\ast}$ of $X$ such that
$\overline{C}_{z}$ is singular and $X_{\mid C_{z}}$ is complete.
If $C_{z}$ is not algebraic, then $X$ is one of the following polynomial vector fields, up to a polynomial automorphism of $\mathbb{C}^{2}$:
\begin{enumerate}
\item
$$
X=\lambda x\, \frac{\partial}{\partial x} + [a(x)y+c(x)]\,\frac{\partial}{\partial y},
$$
with $\lambda/a(0)\in\mathbb{Q}^{+}\setminus \{\mathbb{N}^{+}\cup 1/\mathbb{N}^{+}\}$.
\item
$$X= \, x[n\,f(x^m y^n)+ \alpha] \,\frac{\partial}{\partial x} - y
[m\,f(x^m y^n)+ \beta]\,\frac{\partial}{\partial y},
$$
with $m,n\in \mathbb{N}^{*}$, $f(z)\in z\cdot\mathbb{C}[z]$, $\alpha,\beta\in\mathbb{C}$ such that
$\alpha m - \beta n\in\mathbb{C}^{\ast}$ and
$-\alpha/\beta \in\mathbb{Q}^{+}\setminus\{\mathbb{N}^{+}\cup 1/\mathbb{N}^{+}\}$.
\end{enumerate}
\end{proposition}
\begin{proof}
The fact that $\overline{C}_{z}=C_{z}\cup \{p\}$ implies that $X$ is one of the following vector fields,
up to a polynomial automorphism \cite[p.\,663]{Bustinduy-indiana}:\\
\noindent $(a)$
$$\lambda x\, \frac{\partial}{\partial x} + [a(x)y+b(x)]\,
\frac{\partial}{\partial y},$$
where $a,b\in\mathbb{C}[x]$, and $a(0),\lambda\in\mathbb{C}^{\ast}$.\\
\noindent $(b)$
$$\, x[n\,f(x^m y^n)+ \alpha] \,\frac{\partial}{\partial x} - y
[m\,f(x^m y^n)+ \beta]\,\frac{\partial}{\partial y},$$
with $m,n\in \mathbb{N}^{*}$, $f(z)\in z\cdot\mathbb{C}[z]$,
$\alpha,\beta\in\mathbb{C}$ such that
$\beta/\alpha\in\mathbb{Q}^{-}$ and $\alpha m - \beta n
\in\mathbb{C}^{\ast}$.\\
\noindent $(c)$
$$
x[n\, S + \alpha] \frac{\partial}{\partial x} + \left\{- \frac{ [
n T + m(x^{\ell} y + p(x) )]\,S + \alpha T}{x^{\ell}} \right\}
\frac{\partial}{\partial y},
$$
for $m,n,\ell\in\mathbb{N}^{\ast}$, $\alpha\in\mathbb{C}^{\ast}$,
$p(x)\in \mathbb{C}[x]$ of degree $<{\ell}$, $p(0)\neq{0}$,
$T=\ell x^{\ell}y +xp'(x)$, $S= f(x^m {(x^{\ell}y+ p(x))}^n)$ with
$f(z)\in z\cdot\mathbb{C}[z]$, and where $$[n\,xp'(x) + m\,p(x)] S
+ \alpha x p'(x)\in x^{\ell}\cdot \mathbb{C}[x,y].$$
\noindent
Let us analyze such an $X$ when $\overline{C}_{z}$ is a singular curve.

\noindent $\bullet$ {\it Case} $(a)$. Let us  use well-known
results about singularities of vector fields around $p=(0,0)$
\cite[pp.\,11-16]{Brunella-impa}. Let $\lambda_{1}=\lambda$ and
$\lambda_{2}=a(0)$ be the eigenvalues of the linear part $DX(p)$
of $X$ at $p$ .

If $\lambda_{1}/\lambda_{2}\not\in \mathbb{Q}^{+}$, there are only two separatrices of $\mathcal{F}$ through $p$,
which are smooth and transversal at $p$, which is impossible. If $\lambda_{1}/\lambda_{2}=r/s\in \mathbb{Q}^{+}$, there are two possibilities:

If $r/s\not \in \mathbb{N}^{+}\cup 1/\mathbb{N}^{+}$, according to Poincar\'e's linearization theorem, $X$ is $r z \partial /\partial z + s w \partial /\partial w $
in certain coordinates around $p$.
Then $z^{s}/w^{r}$ is a local first integral, and this possibility can occur.

If $r/s\in \mathbb{N}^{+}\cup 1/\mathbb{N}^{+}$, according to Poincar\'e-Dulac's normal form theorem,
$X$ is $z \partial /\partial z + (n w + \epsilon z^{n})\partial /
\partial w $, with $\epsilon\in\{0,1\}$, $n=r/s$ or $s/r$ $\in\mathbb{N}^{+}$ around $p$. If $\epsilon=0$, $z^{n}/w$ is a local
first integral. Then all the separatrices through $p$ are smooth, which is not possible. If $\epsilon=1$, $z \,e^{-w/z^{n}}$ is a local
first integral. Thus there are no separatrices different from $\{x=0\}$, which is also impossible.
Therefore, $\lambda/a(0)\in\mathbb{Q}^{+}\setminus \{\mathbb{N}^{+}\cup 1/\mathbb{N}^{+}\}$.\\
\noindent $\bullet$ {\it Case} $(b)$. The eigenvalues of $DX(p)$
are $\lambda_{1}=\alpha$ and $\lambda_{2}=-\beta$. Then $\lambda_{1}/\lambda_{2} \in \mathbb{Q}^{+}$.
One analyzes as in case $(a)$ that $\lambda_{1}/\lambda_{2}\in\mathbb{Q}^{+}\setminus \{\mathbb{N}^{+}\cup 1/\mathbb{N}^{+}\}$.\\
\noindent $\bullet$ {\it Case} $(c)$. According to
\cite[p.\,649]{Bustinduy-indiana}, if $H$ is the regular covering
map from $u\neq {0}$ to $x\neq{0}$, $(u,v)\mapsto
(x,y)=H(u,v)=(u^n,{u^{-(m+n\ell)}} [v-u^m p(u^n)]),$
$$
H^{\ast} X=[f(v^n)+ \alpha / n]u \frac{\partial}{\partial u} +
[\alpha m/n]v \frac{\partial}{\partial v}.
$$
Then $X$ has only one zero $p$, which is on $\{x=0\}$ (invariant by $X$). Working with the expression of $X$  one obtains that $DX(p)$ has eigenvalues $\lambda_{1}=\alpha$ and
$\lambda_{2}=-\alpha \ell$. Then $\lambda_{1}/\lambda_{2}=-1/\ell\not\in\mathbb{Q}^{+}$, and there are only two separatrices  through $p$,
which are smooth and transversal at $p$, which is not possible. Hence $(c)$ does not occur.
\end{proof}
Let us see in the following proposition the analytic version of Proposition\,\ref{P6}.
\begin{proposition}\label{P7}
Let $C_{z}$ be a proper trajectory of type $\mathbb{C}^{\ast}$ of $X$ such that
$\overline{C}_{z}$ is singular and $X_{\mid C_{z}}$ is complete.
If $C_{z}$ is not algebraic,
$$
X =r x \,\dfrac{\partial}{\partial x} + s y
\,\dfrac{\partial}{\partial y},
$$
with $r,s\in\mathbb{N}^{+}$, $r\cdot s\neq{1}$ and $(r,s)=1$, up to an analytic automorphism of
$\mathbb{C}^{2}$.
\end{proposition}
\begin{proof}

\medskip

By Proposition\,\ref{P6}, $X$ is complete, of type $\mathbb{C}^{\ast}$, and with only one zero that is moreover
the topological boundary in $\mathbb{C}^{2}$ of any trajectory of $X$.
The proof follows from \cite[p.\,530]{Suzuki-anales}.
\end{proof}
According to Proposition\,\ref{P7}, modulo an analytic automorphism of $\mathbb{C}^{2}$, $C_z$ is contained in a level set of $y^{r}/x^{s}$, which is a first integral of $X$.
Therefore, $\overline{C}_{z}=\{y^{r}-ax^{s}=0\}$, being $p=(0,0)$ and $a\neq{0}$.
The condition  $(r,s)=1$ allows us to assume $pr-qs=1$ for $p$, $q\in\mathbb{N}^{+}$. Let us
consider $x=v^{q}u^{r}$ and $y=v^{p}u^{s}$ with $u$, $v\in\mathbb{C}$. It holds
$$y^{r}-ax^{s}=v^{qs+1}u^{sr} - a
v^{qs}v^{rs}=u^{rs}v^{qs}(v-a).$$
Hence, it is enough to take a surjective dominating map from $\mathbb{C}^{2}$ to $\mathbb{C}^{2}\setminus \{v=a\}$, and compose it with
$\gamma:\mathbb{C}^{2}\setminus \{v=a\}\to \mathbb{C}^{2}\setminus
C_{z}$ defined as $(u,v) \mapsto \gamma(u,v)=(v^{q}u^{r},v^{p}u^{s})$, to obtain $\Gamma$.

Therefore, we have finished the proof of Theorem\,\ref{T5}.

\begin{remark}\em
Note that for a proper trajectory $C_{z}$ of type $\mathbb{C}^{\ast}$ such that $\overline{C}_{z}=C_{z}\cup \{p\}$ with $X(p) = 0$,
$\mathbb{C}^{2}\setminus C_{z}$ is not a manifold. Theorem\,\ref{T5} implies that if $\overline{C}_{z}$ is singular, and moreover $X$ is polynomial with $X_{\mid C_{z}}$ complete,
$\mathbb{C}^{2}\setminus C_{z}$ is a holomorphically dominable set by $\mathbb{C}^{2}$. \em
\end{remark}

\bibliographystyle{plain}
\bibliographystyle{amsalpha}
%\bibliography{albib}
\def\cprime{$'$} \def\cprime{$'$}

\end{document}